\documentclass[11pt]{amsart}
\usepackage{amssymb,amsfonts,graphicx,amsmath,amsthm,amstext,latexsym,amsfonts}
\usepackage{graphicx,epsfig}
\usepackage{bbm}
\usepackage{t1enc}
\usepackage{mathrsfs}  
\usepackage{subfig}
\usepackage{hyperref}
\usepackage{a4wide}
 \usepackage{bm}
\usepackage{tikz}
\usetikzlibrary{intersections}
\allowdisplaybreaks
\theoremstyle{plain}

\numberwithin{equation}{section}
\newtheorem{theorem}{Theorem}[section]
\newtheorem{lemma}[theorem]{Lemma}
\newtheorem{definition}[theorem]{Definition}

\usepackage{color}
\definecolor{darkred}{rgb}{0.8,0,0}
\definecolor{darkblue}{rgb}{0,0,0.7}
\definecolor{darkgreen}{rgb}{0,0.4,0}

\newcommand{\eps}{\varepsilon}

\newcommand{\R}{{\mathbb R}}

\newcommand{\W}{{\mathcal W}}

\newcommand{\norm}[1]{\left \| #1 \right \|}

\newcommand{\un}{{\rm 1\kern -2.5pt l}}

\newcommand{\xxi}{{\mbox{\boldmath$\xi$}}}

\def\u{\mathbf{u}}

\def\U{\mathbf{U}}
\def\V{\mathbf{V}}

\def\b{\mathbf{b}}

\def\xxi{\boldsymbol{\xi}}

\def\eps{\varepsilon}
\def\R{{\mathbb R}}

\def\eps{\varepsilon}
\def\R{{\mathbb R}}

\def\e{{\mathcal F}}

\def\F{{\mathcal F}}

\def\W{{\mathcal W}}

\def\u{\mathbf{u}}
\def\ue{\mathbf{u}_\eps}
\def\epk{{\eps_k}}
\def\uek{\mathbf{u}_\epk}
\def\wu{\mathbf{\widetilde u}}
\def\wv{\mathbf{\widetilde v}}
\def\v{\mathbf{v}}
\def\e{\mathbf{e}}
\def\z{\mathbf{z}}
\def\f{{\bf f}}
\def\v{{\bf v}}
\def\w{{\bf w}}
\def\x{{\bf x}}
\def\y{{\bf y}}
\def\z{{\bf z}}

\def\pt{\partial_t}
\def\ptt{\partial_t^2}

\newcommand{\weak}{\rightharpoonup}
\newenvironment{Assumptions}
{%

\begin{enumerate}}%
{\end{enumerate}}

\newenvironment{Definitions}
{%

\begin{enumerate}}%
{\end{enumerate}}

\newcommand{\bfeta}{{\mbox{\boldmath$\eta$}}}

  \title[Nonlinear  Peridynamic]{Wellposedness of a Nonlinear  Peridynamic Model}

\author{Giuseppe Maria Coclite}\author{Serena Dipierro}\author{Francesco Maddalena}\author{Enrico Valdinoci}
\address[Giuseppe Maria Coclite and Francesco Maddalena]{\newline
Department of Mechanics, Mathematics and Management, Polytechnic of Bari,
Via E.~Orabona 4, 70125 Bari, Italy.}
\email[]{giuseppemaria.coclite@poliba.it, francesco.maddalena@poliba.it}

\address[Serena Dipierro]{\newline
School of Mathematics and Statistics,
University of Western Australia,
35 Stirling Highway,
Crawley, WA 6009, Australia, and
Department of Mathematics, University of Milan,
Via Saldini 50, 20133 Milan, Italy.}
\email[]{serydipierro@yahoo.it}

\address[Enrico Valdinoci]{\newline
School of Mathematics and Statistics,
University of Melbourne,
813 Swanston Street, Parkville VIC 3010, Australia,
and
School of Mathematics and Statistics,
University of Western Australia,
35 Stirling Highway,
Crawley, WA 6009, Australia, and
Department of Mathematics, University of Milan,
Via Saldini 50, 20133 Milan, Italy, and
Institute for Applied Mathematics and Information Technologies,
Via Adolfo Ferrata 1, 27100 Pavia, Italy.
}
\email[]{enrico.valdinoci@gmail.com}

\date{\today}


\thanks{The  authors are members of the Gruppo Nazionale per l'Analisi Matematica, la Probabilit\`a e le loro Applicazioni (GNAMPA) of the Istituto Nazionale di Alta Matematica (INdAM).
SD and EV are supported by the 
Australian Research Council
Discovery Project grant ``Nonlocal Equations at Work'' (NEW)}

\begin{document}
 \maketitle
\begin{abstract}
We consider an evolution equation inspired by a model
in peridynamics, with a singular pairwise interaction force term,
and we give global in time existence, uniqueness and stability results for the Cauchy problem.
\end{abstract}


\section{Introduction}

The exceptional achievements in materials science and the new
technological issues ask for a continuous
deepening of our understanding of the materials behavior.
Since the end of the sixties, the need to enlarge the framework of continuum mechanics in order to keep 
track of nonlocal effects was  recognized by many researchers.
More precisely, E. Kr\"oner, D. G. B. Edelen, A. C. Eringen and I. A. Kunin
(see~\cite{Er, EE, K, Ku}) laid the foundations for a
comprehensive theoretical treatment of \textit{nonlocal} mechanics.\\
More recently, in~\cite{S1} S. A. Silling  introduced \textit{peridynamics},
as a nonlocal elasticity theory: a continuum theory avoiding
spatial derivatives and incorporating spatial nonlocality (see also \cite{S, SEWX, SL, SL1}). 
Peridynamics allows  to model nonlocal interactions
through 
long-range forces, and it
is believed to be suited for the
description of a large class of physical 
phenomena which escape a classical description of continuum mechanics based on partial differential equations.
In particular, the theory of
peridynamics seems to offer a promising framework to model phenomena such as damage and fracture in solids,
evolution of phase boundaries in phase transformations,
defects, dislocations, etc.
\medskip

We now introduce the mathematical framework in which we work.
Let~$\Omega\subset \R^N$ be the rest configuration of a material body
endowed with a mass density~$\rho:\Omega\times [0,T]\rightarrow \R_+$, and
let~$\u:\Omega\times [0,T]\rightarrow \R^N$ be the displacement field
assigning at the particle having position~$\x\in \Omega$
at time~$t=0$ the new position~$\x+\u(\x, t)$ at time~$t$.
The crucial assumption of peridynamics relies in postulating the
existence of a long range internal force field, in place of the classical contact forces.
Therefore the evolution of the material body is ruled by  
the following nonlocal version of the linear momentum balance: 
\begin{align}
\label{bl}
&\rho(\x,t)\ptt \u(\x,t)=\int_{V_{\x}\cap \Omega}\f(\x,\x', \u(\x,t),\u(\x',t),t)\,d\x'+\b(\x,t),\\
\label{ic}
&\u(\cdot,0)=\u_0,\qquad\pt \u(\cdot,0)=\v_0,
\end{align}
where $V_\x\subset \R^N$ is a measurable subset with
\begin{equation}\label{delta}
\x\in V_\x,\qquad{\rm diam}(V_\x)\ge\delta>0
\end{equation}
and $\b(x,t)$ represents the external body force field.

Let us make some comments on~\eqref{bl}. Notice that the internal contact forces,
condensed in the Cauchy stress tensor 
and representing the fundamental concept in classical continuum
mechanics, here are replaced by the pairwise force 
function~$\f$ which can be thought as the density of the interaction between
the particle at~$\x$ and all the particles~$\x'$
belonging to the region~$V_\x$ (one can also assume that~$V_\x=\Omega$).
Then, whereas in the classical context we have to face with partial differential equations and the evolution problem is an initial boundary value problem, in the present context we have an integro-differential equation.\\

The fundamental issue in this setting, which constitutes the core of
the \textit{mathematical-physics wellposedness},
relies in selecting the force field~$\f$ in such a way that it satisfies
the general principles of mechanics, to capture the essential
features of the material behavior, to deliver a well posed
mathematical problem. 
In this framework, we study the Cauchy problem for an unbounded domain,
under general enough assumptions on the force field~$\mathbf f$.\medskip

Due to the balance of linear and angular momentum,
the pairwise force function~$\f$ has the direction of the vector joining~$
\x+\u(\x,t)$ to~$\x'+\u(\x',t)$, therefore we can write
\begin{equation*}
\begin{split}
\f(\x,\x',\u(\x,t),\u(\x',t),t)=&f(\x,\x',\u(\x,t),\u(\x',t),t)\,\e, \\
{\mbox{where }}\quad \e=&\frac{(\x'+\u(\x',t))-
(\x+\u(\x,t))}{|(\x'+\u(\x',t))-(\x+\u(\x,t))|}.
\end{split}
\end{equation*}
Furthermore, assuming the invariance with respect to rigid motions
and neglecting time dependence for the internal forces, we get
\begin{equation}
\label{dir}
\f (\x,\x',\u(\x,t),\u(\x',t))= \f (\x,\x', \u(\x,t)-\u(\x',t))
\end{equation}
and the
Newton law of \textit{actio et reactio} delivers
\begin{equation}
\label{Nl}
{\f} (\x,\x',-{\bfeta})=-{\f}(\x,\x',{\bfeta}).
\end{equation}
Coherently with the literature on peridynamics in this section we will often use the notation
\begin{equation*}
\xxi:=\x'-\x, \qquad\bfeta:=\u(\x',\cdot)-\u(\x,\cdot).
\end{equation*}

Let us continue by providing some examples of $\f$ in specific cases.
For a \textit{Linear Elastic Material}, we have
\begin{equation}
\label{linel}
\f(\x,\x',\bfeta):=\f_0(\x,\x')+{\mathbf C}(\x,\x'){\bfeta}
\end{equation}
By recalling \eqref{dir}, it is readily seen that the tensor~$\mathbf C$ takes the form
\begin{equation}
\label{stiff}
{\mathbf C}(\x,\x'):=\lambda(|\xxi|)\xxi\otimes\xxi
\end{equation}
where $\lambda:\R_+\rightarrow\R$ is a measurable function
such that $\lambda(r)=0$ for $r\geq \delta$ (recall that~$\delta$
is the quantity introduced in~\eqref{delta}).
The tensor~$\mathbf C$ determines the specific material and
depends on~$N$ and~$\delta$.

In the case of a \textit{Nonlinear Elastic Material} one can assume
\begin{equation}
\label{nonlin}
{\f(\x,\x',\bfeta)}:=\begin{cases}
{\varphi}\,\left(\frac{| \xxi+\bfeta|}{|\xxi|}-{1}\right)\e,&\quad \text{if}\>|\xxi|<\delta,\\
\mathbf 0,&\quad \text{if}\>|\xxi|\ge\delta,
\end{cases}
\end{equation}
where the function $\varphi$ depends
on~$s:=\left(\frac{| \xxi+\bfeta|}{|\xxi|}-{1}\right)$, which represents the 
\textit{bond stretch}, i.e. the relative change of the length of a pairwise bond.
 
The main mathematical problems in peridynamics concern well-posedness and regularity for the integral-differential
equation \eqref{bl} under the particular choice of $\f$ related to the material behavior and the analysis of the limit for vanishing nonlocality, i.e. the issue of characterizing the solution as well as the problem as $\delta\rightarrow 0$.

In the linear elastic case, well-posedness and regularity
were established in~\cite{EW}, though the problem of the
limit as~$\delta\rightarrow 0$ is still largely open.
In particular, assuming that~$\u_0,\v_0\in L^p(\Omega)$
and~$\mathbf b\in L^1(0,T;L^p(\Omega))$, the authors prove
the existence of a unique solution~$\u\in C^1([0,T];L^p(\Omega))$
to the initial-value problem~\eqref{bl}-\eqref{ic} in the linear case~\eqref{linel}.

In the case of nonlinear elasticity, the situation is more involved
and at the present time few results are known \cite{DZ, ELP, EP1, EP, EP2, EW, MD}
(the situation is even more complicated
in the classical nonlinear elastodynamics).
According to the authors of~\cite{EP}, 
the main known results
can be represented by the two theorems below related to the peridynamic operator
\begin{equation*}
(K\u)(\x):= \int_{\Omega\cap B_\delta(\x)}\f(\x'-\x,\u(\x')-\u(\x))\,d\x',
\end{equation*}
where $B_\delta(\x)$ denotes the $N$-dimensional open ball centered in $\x$ of radius $\delta$.

The first result that we recall is the following:

\begin{theorem}\label{Tv1}{\rm (\cite{EP})}
Let $\u_0,\v_0\in C(\overline{\Omega})^N$ and ${\mathbf b}\in C([0,T];
C(\overline{\Omega})^N)$. Assume that~$\f: \overline{B_\delta(\mathbf 0)}\times
\R^N\rightarrow \R^N$ is continuous and that there exists a
nonnegative function~$\ell\in L^1(B_\delta(\mathbf 0))$
such that for all $\bm\xi\in \R^N$ with $\vert \xxi\vert\leq \delta$ and $\bm\eta,\bm\eta'$ there holds
$$\vert \f(\bm\xi,\bm\eta')-\f(\bm\xi,\bm\eta)\vert\leq\,\ell(\xxi)\vert\bm\eta'-\bm\eta\vert.$$ 
Then, the peridynamic operator $K:C(\overline{\Omega})^N\rightarrow\,C
(\overline{\Omega})^N$ is well-defined and Lipschitz-continuous,
and the initial-value problem \eqref{bl}-\eqref{ic} is globally well-posed
and there exists a solution~$\u\in C^2([0,T];C(\overline{\Omega})^N)$.
\end{theorem}

Notice that, as remarked by the authors of~\cite{EP},  
global Lipschitz-continuity of the pairwise force function
with respect to~$\bm\eta$ is quite a
restrictive assumption since it implies linearly bounded growth.

The next result that we recall
regards the existence of weak solutions.
For this, before stating the theorem, we need some definitions.
We denote by~$(\cdot,\cdot)$ the inner product in~$L^2(\Omega)^N$
and we consider a Banach space~$X\subset L^2(\Omega)^N$.
Moreover, $K:X\rightarrow X^*$ is the energetic extension of
the peridynamic operator, namely
$$\langle K\w,\z\rangle:=
\frac{1}{2}\int_{\Omega\times\Omega}
a(\vert\bm \xi\vert, \vert\bm\eta\vert)(\w(\x')-\w(\x))\cdot(\z(\x')-\z(\x))\,d(\x',\x),$$
for every~$\w$, $\z\in X$.

With this notation, we say that $\u:[0,T]\rightarrow X$ satisfying
the initial conditions~\eqref{ic} is a weak solution for~\eqref{bl} if,
for every~$\varphi\in\,C_0^\infty(0,T)$ and every~$\z\in X$, 
$$-\int_0^T\pt\varphi'(t)(\pt\u(t),\z)\,dt-\int_0^T\langle K\u(t),\z\rangle \varphi(t)\,dt
=\int_0^T({\mathbf b}(t),\z)\varphi(t)\,dt.$$

In the following, the basic function space is the Sobolev-Slobodeckij space~$X= W^{\sigma,p}(\Omega)^N$ with~$0<\sigma<1$ and~$2\leq p<+\infty$,
and~$C_w([0,T]; X)$ denotes the space of the functions~$v:[0,T]\rightarrow X$
which are continuous with respect to the weak convergence in~$X$. In this setting, we have:

\begin{theorem}\label{Tv2}{\rm (\cite{EP2})}
Assume that~${\mathbf b}\in L^1(0,T;L^2(\Omega)^N)$,
$\u_0\in W^{\sigma,p}(\Omega)^N$, $\v_0\in L^2(\Omega)^N$
and the pairwise force function~$\f$
satisfies suitable growth and regularity conditions.

Then, there exists a function $\u:[0,T]\rightarrow W^{\sigma,p}(\Omega)^N$
with
$$\u\in C_w([0,T]; W^{\sigma,p}(\Omega)^N),\:\:
\pt\u\in C_w([0,T]; L^2(\Omega)^N),\:\:
\ptt\u\in L^1(0,T; (W^{\sigma,p}(\Omega)^N)^*),$$
such that
$$\ptt\u-K\u=\mathbf b\:\:\:\hbox{in}\:\:L^1(0,T; (W^{\sigma,p}(\Omega)^N)^*)$$
and $\u(\cdot,0)=\u_0$ in $W^{\sigma,p}(\Omega)^N$, $\pt\u(\cdot,0)=\v_0$ in 
$L^2(\Omega)^N$.
\end{theorem}

The goal of this paper is to provide:
\begin{itemize}
\item {\bf existence} results in the spirit of Theorems~\ref{Tv1}
and~\ref{Tv2}, but which are valid for singular and non-Lipschitz interaction force,
\item {\bf uniqueness} and {\bf stability} results.
\end{itemize}
The complete description of the mathematical setting in which
we work will be given in the forthcoming Section~\ref{sec:math}, but, for simplicity,
we mention here that the interaction force  we can take into account comprises, among the
others, examples of the form
$$ \f(\bm\xi,\bm\eta)=\frac{|\bm\eta|^{p-2}\,\bm\eta}{|\bm\xi|^{N+\alpha p}}+\bm{\psi}(\bm\xi,\bm\eta),$$
where~$p\ge2$, $\alpha\in(0,1)$ and $\bm\psi$ plays the role of a ``sufficiently smooth perturbation'', e.g.
$$\bm\psi(\bm\xi,\bm\eta)=\sum_{i=1}^N\sin\bm\xi_i\,\cos\bm\eta_i.$$
In this framework, we establish that 
\medskip

\noindent\fbox{ 
\parbox{\textwidth}{
{\em
for every  initial datum with finite energy,
the Cauchy problem admits at least one weak solution 
whose energy at time~$t$ is bounded by the initial one}. }}
\medskip

A precise statement for this existence result will be given in Theorem~\ref{th:existence}.

In the case~$p=2$, we also provide
a uniqueness and stability result. Namely, 
\medskip

\noindent\fbox{ 
\parbox{\textwidth}{
{\it
if $\u$ and $\wu$ are weak solutions
with finite energy initial data, then the quantity
$$ \norm{\pt\u(\cdot,t)-\pt\wu(\cdot,t)}_{L^2(\R^N)}^2+
\int_{\R^N}\int_{B_\delta(\mathbf 0)}
\frac{|\u(\x,t)-\u(\x-\y,t)-\wu(\x,t)+\wu(\x-\y,t)|^2}{|\y|^{N+2\alpha}}\,d\x\, d\y$$
is bounded, up to constants, by the same quantities at the initial time,
multiplied by an exponential in time}. 
}}\medskip

In particular, if the initial data coincide,
the two solutions must coincide as well. A detailed statement for this result
will be given in Theorem~\ref{th:stability}.\medskip

It is interesting to point out that
the functional spaces in which we work allow, in principle,
singular functions (see Appendix~B in~\cite{PAT}).

The rest of the paper is organized as follows. In Section \ref{sec:math} we 
state precisely the problem that we study, the assumptions and the main results.
Section \ref{sec:exist} is devoted to the proof of the existence of weak solutions. The uniqueness and stability of such solutions is proved in Section \ref{sec:stab}.

\section{Statement of the problem and main results}
\label{sec:math}

This section is devoted to the rigorous mathematical
formulation of the problem, the definition of weak solutions,
and the statements of the main results of the paper.
To keep  the analysis of the mathematical problem clear we assume $\b\equiv\mathbf 0$.

\subsection{Set-up of the problem and assumptions}
\label{subsec:math.1}
We consider the  Cauchy problem
\begin{equation}
\label{eq:CP}
\begin{cases}
\ptt \u(\x,t)=(K\u(\cdot,t))(\x),&\quad \x\in \R^N,\>t>0,\\
\u(\x,0)=\u_0(\x),\>\pt \u(\x,0)=\v_0(\x),&\quad \x\in\R^N,
\end{cases}
\end{equation}
where
\begin{equation}
\label{eq:operator}
(K\u)(\x):= \int_{B_\delta(\x)}\f(\x'-\x,\u(\x')-\u(\x))\,d\x',\quad {\mbox{ for every }}\,\x\in\R^N,
\end{equation}
for a given $\delta>0$.
Here, the $\R^N$ valued  function $\f$ is defined on the set 
\begin{equation*}
\Omega:=(\R^{N}\setminus\{\mathbf 0\})\times\R^N
\end{equation*}
and we assume that
\begin{Assumptions}
\item \label{ass:f.1} $\f\in C^1(\Omega;\R^N)$;
\item \label{ass:f.2} $\f(-\y,-\u)=-\f(\y,\u),$ for every $(\y,\,\u)\in\Omega\times \R^N$;
\item \label{ass:f.5} the material is hyperelastic, i.e., there exists a function $\Phi\in C^2(\Omega)$ such that
\begin{equation*}
 \f=\nabla_{\u}\Phi,\qquad \Phi(\y,\u)=\kappa \frac{|\u|^p}{|\y|^{N+\alpha p}}+
 \Psi(\y,\u),\quad {\mbox{ for every }}\, (\y,\,\u)\in\Omega,
\end{equation*}
where $\kappa,\, p,\,\alpha$ are constants such that
\begin{equation*}
\kappa>0,\qquad 0<\alpha<1,\qquad   p\ge 2,
\end{equation*}
and 
\begin{align*}
& \Psi(\y,\mathbf 0)=0\le \Psi(\y,\u),\\&
|\nabla _\u \Psi(\y,\u)|, |D^2_{\u}\Psi(\y,\u)|\le g(\y),\quad {\mbox{ for every }}\,(\y,\,\u)\in\Omega,
\end{align*}
for some nonnegative function~$g\in L^2_{\rm{loc}}(\R^N)$.
\end{Assumptions}

We notice that if~$\Psi(-\y,-\u)=\Psi(\y,\u)$, then~\ref{ass:f.2} holds true. In addition,
Assumption \ref{ass:f.5} can be easily generalized to the anisotropic
case by taking
\begin{equation*}
\Phi(\y,\u):=({\mathbb K}\u\cdot\u) \frac{|\u|^{p-2}}{|\y|^{N+\alpha p}}
+ \Psi(\y,\u),\quad {\mbox{ for every }}\,(\y,\,\u)\in\Omega\times \R^N,
\end{equation*}
where ${\mathbb K}\in\R^{N\times N}$ is a positive definite matrix.
Here, we stick to assumption~\ref{ass:f.5} for the sake of simplicity.

We observe that \ref{ass:f.1}, \ref{ass:f.2}, \ref{ass:f.5} are the only constitutive assumptions characterizing the peridynamic model and they are sufficient to prove global well-posedness of the problem,
as it is shown in Theorems \ref{th:existence} and \ref{th:stability} below.
We emphasize that here, in contrast to classical (local) elastodynamics neither polyconvexity or null condition (see \cite{B,Sid}) are required to guarantee existence of global (in time) solutions.

Setting $\y=\x'-\x$ and using \ref{ass:f.2},
we will often rewrite the operator $K$ in~\eqref{eq:operator}
as follows
\begin{equation}
\label{eq:operator1}
(K\u)(\x)= -\int_{B_\delta(\mathbf 0)}\f(\y,\u(\x)-\u(\x-\y))d\y,\quad {\mbox{ for every }}\,\x\in\R^N.
\end{equation}
Moreover, due to \ref{ass:f.5}, we can also rewrite
\begin{equation}
\label{eq:operator2}
 \f(\y,\u)=\kappa p \frac{\u|\u|^{p-2}}{|\y|^{N+\alpha p}}+
 \nabla_{\u}\Psi(\y,\u),\quad {\mbox{ for every }}\, (\y,\,\u)\in\Omega.
\end{equation}
As a consequence, in virtue of~\eqref{eq:operator1} and~\eqref{eq:operator2},
we have that
\begin{equation}
\label{eq:operator3}
\begin{split}
(K\u)(\x)= &-\kappa p \int_{B_\delta(\mathbf 0)}\frac{(\u(\x)-\u(\x-\y))
|\u(\x)-\u(\x-\y)|^{p-2}}{|\y|^{N+\alpha p}}\,d\y
\\& -\int_{B_\delta(\mathbf 0)}\nabla_{\u}\Psi(\y,\u(\x)-\u(\x-\y))\,d\y.
\end{split}
\end{equation}
Furthermore, the energy associated to \eqref{eq:CP} is
\begin{equation}
\label{eq:enert}
E[\u](t):= \frac{\norm{\pt \u(\cdot,t)}^2_{L^2(\R^N)}}{2}
+\frac{1}{2}\int_{\R^N}\int_{B_\delta(\mathbf 0)}
\Phi(\y,\u(\x,t)-\u(\x-\y,t))\,d\x\, d\y,
\end{equation}
see the forthcoming proof of Lemma~\ref{lm:enest} below. 
Also, the equation in~\eqref{eq:CP} is the Euler-Lagrange equation
of the action functional
\begin{equation}
\label{eq:action}
\u\mapsto \int_0^\infty\int_{\R^N}\left(\frac{(\pt \u)^2}{2}
-\frac{1}{2}\int_{B_\delta(\mathbf 0)}\Phi(\y,\u(\x,t)-\u(\x-\y,t))\,d\y\right)\,dt\,d\x.
\end{equation}

\subsection{Fractional Sobolev spaces}
\label{subsec:math.2}

Let $\alpha$ and $p$ be the parameters introduced in~\ref{ass:f.5}.
We remind that the fractional Sobolev space $W^{\alpha,p}$ is defined
through the norm
\begin{equation*}
\norm{\u}_{W^{\alpha,p}(\R^N;\R^N)}:=
\left(\int_{\R^N} |\u|^pd\x+\int_{\R^N}\int_{\R^N}
\frac{|\u(\x)-\u(\x-\y)|^p}{|\y|^{N+\alpha p}}\,d\x\, d\y\right)^{1/p}
\end{equation*}
(see \cite[Section 2]{DPV}),
and the following compact embedding holds
\begin{equation*}
W^{\alpha,p}(\R^N;\R^N)\hookrightarrow\hookrightarrow L^q_{\rm{loc}}(\R^N;\R^N),
\qquad 1\le q\le p,
\end{equation*}
(see \cite[Theorem 7.1]{DPV}).

In this paper,
we use a slight modification of the fractional Sobolev space~$W^{\alpha,p}$.
Namely, we consider the space~$\W$ defined through the norm
\begin{equation}
\label{eq:Sobolev1}
\norm{\u}_{\W}:= \norm{\u}_{L^2(\R^N;\R^N)}+
\left(\int_{\R^N}\int_{B_\delta(\mathbf 0)}
\frac{|\u(\x)-\u(\x-\y)|^p}{|\y|^{N+\alpha p}}\,d\x\, d\y\right)^{1/p}.
\end{equation}
Then, the following compact embedding  can be proved:

\begin{lemma}
\label{eq:Sobolev2}
$\W\hookrightarrow\hookrightarrow L^2_{\rm{loc}}(\R^N;\R^N)$.
\end{lemma}

\begin{proof} One can
argue as in~\cite[Theorem 7.1]{DPV}, or make the following observations.
We suppose~$p>2$ (the case~$p=2$ following directly from~\cite[Theorem 7.1]{DPV}).
Fix~$\lambda\in(0,\alpha)$ and a bounded domain~${\mathcal{K}}\subset\R^N$. We write
\begin{align*} N+2(\alpha-\lambda) &= \frac{2(N+\alpha p)}{p}+
\frac{p(N+2(\alpha-\lambda))-2(N+\alpha p)}{p}\\&
=\frac{2(N+\alpha p)}{p}+\frac{(p-2)N}{p}-2\lambda .\end{align*}
Therefore, using the H\"older inequality with exponents~$\frac{p}2$ and~$\frac{p}{p-2}$,
we find that, for any~${\mathbf z}\in\R^n$,
\begin{align*}
& \int_{B_{\delta/2}(\mathbf z)}\int_{B_{\delta/2}(\mathbf z)}
\frac{|\u(\x)-\u(\y)|^2}{|\x-\y|^{N+2(\alpha-\lambda) }}\,d\x\, d\y
=\int_{B_{\delta/2}(\mathbf z)}\int_{B_{\delta/2}(\mathbf z)}
\frac{|\u(\x)-\u(\y)|^2}{|\x-\y|^{\frac{2(N+\alpha p)}{p}}}
\frac{d\x\, d\y}{|\x-\y|^{\frac{(p-2)N}{p}-2\lambda}}
\\ &\qquad\le
\left(
\int_{B_{\delta/2}(\mathbf z)}\int_{B_{\delta/2}(\mathbf z)}
\frac{|\u(\x)-\u(\y)|^p}{|\x-\y|^{N+\alpha p}}
\,d\x\, d\y
\right)^{\frac2p}\left(
\int_{B_{\delta/2}(\mathbf z)}\int_{B_{\delta/2}(\mathbf z)}
\frac{d\x\, d\y}{|\x-\y|^{N-\frac{2\lambda p}{p-2}}}\right)^{\frac{p-2}p}
\\ &\qquad\leq C\,
\left(
\int_{\R^N}\int_{B_{\delta}(\mathbf 0)}
\frac{|\u(\x)-\u(\x-\y)|^p}{| \y|^{N+\alpha p}}
\,d\x\, d\y
\right)^{\frac2p},
\end{align*}
for some~$C>0$, possibly depending on~$N$, $\alpha$, $p$
and~$\lambda$. Accordingly, if a family of functions is bounded
in~$\W$, then each component is bounded in~$W^{\alpha-\lambda,2}(B_{\delta/2}(\mathbf z))$,
for any~$\mathbf z\in\R^N$, and so, by~\cite[Theorem 7.1]{DPV},
we obtain compactness in~$L^2(B_{\delta/2}(\mathbf z))$.
Arguing component by component and covering~$\mathcal{K}$ with a finite number
of balls of radius~$\delta/2$, we obtain the desired compactness
in~$L^2({\mathcal{K}};\R^N)$.
\end{proof}

\begin{lemma}
\label{lm:Sobolev1} 
For every $\u$, $\v\in \W$ we have that
\begin{equation}
\label{eq:Sobolev3}
(K\u)\,\v\in L^1(\R^N).
\end{equation}
Moreover, for every  sequence $\{\u_n\}_n\subset \W$ and $\u\in \W$, if
\begin{equation}
\label{eq:Sobolev4}
\u_n\weak \u \quad\text{weakly in $\W$},
\end{equation}
then 
\begin{equation}
\label{eq:Sobolev5}
K\u_n\to K\u \quad\text{ in the sense of distributions on $\R^N$},
\end{equation}
as~$n\to+\infty$.
\end{lemma}

\begin{proof}
We first prove~\eqref{eq:Sobolev3}. 
To this aim, we let $\u,\,\v\in \W$.
Recalling~\eqref{eq:operator3} and~\ref{ass:f.5}, and
using the H\"older inequality, we have that
\begin{align*}
\int_{\R^N}& (K\u)(\x)\v(\x)\,d\x\\
=&-\kappa p \int_{\R^N}\int_{B_\delta(\mathbf 0)}
\frac{(\u(\x)-\u(\x-\y))|\u(\x)-\u(\x-\y)|^{p-2}}{|\y|^{N+\alpha p}}\v(\x)\,d\x\, d\y\\
&\quad -\int_{\R^N}\int_{B_\delta(\mathbf 0)}\nabla_{\u}
\Psi(\y,\u(\x)-\u(\x-\y))\v(\x)\,d\x\, d\y\\
=&-\frac{\kappa p}{2} \int_{\R^N}\int_{B_\delta(\mathbf 0)}
\frac{(\u(\x)-\u(\x-\y))|\u(\x)-\u(\x-\y)|^{p-2}}{|\y|^{N+\alpha p}}\v(\x)\,d\x\, d\y\\
&\quad -\frac{\kappa p}{2} \int_{\R^N}\int_{B_\delta(\mathbf 0)}
\frac{(\u(\z+\y)-\u(\z))|\u(\z+\y)-\u(\z)|^{p-2}}{|\y|^{N+\alpha p}}\v(\z+\y)\,d\z\, d\y\\
&\quad -\int_{\R^N}\int_{B_\delta(\mathbf 0)}\nabla_{\u}
\Psi(\y,\u(\x)-\u(\x-\y))\v(\x)\,d\x\, d\y\\
=&-\frac{\kappa p}{2} \int_{\R^N}\int_{B_\delta(\mathbf 0)}\frac{(\u(\x)-\u(\x-\y))
|\u(\x)-\u(\x-\y)|^{p-2}}{|\y|^{N+\alpha p}}(\v(\x)-\v(\x-\y))\,d\x\, d\y\\
&\quad -\int_{\R^N}\int_{B_\delta(\mathbf 0)}
\nabla_{\u}\Psi(\y,\u(\x)-\u(\x-\y))\v(\x)\,d\x\, d\y\\
\le&\frac{\kappa p}{2} \int_{\R^N}\int_{B_\delta(\mathbf 0)}
\frac{|\u(\x)-\u(\x-\y)|^{p-1}}{|\y|^{(N+\alpha p)/p'}}\frac{|\v(\x)-\v(\x-\y)|}{
|\y|^{(N+\alpha p)/p}}\,d\x\, d\y\\
&\quad +\int_{\R^N}\int_{B_\delta(\mathbf 0)}g(\y)|\u(\x)-\u(\x-\y)||\v(\x)|\,d\x\, d\y\\
\le&\frac{\kappa p}{2} \left(\int_{\R^N}\int_{B_\delta(\mathbf 0)}
\frac{|\u(\x)-\u(\x-\y)|^{p}}{|\y|^{N+\alpha p}}\,d\x\, d\y\right)^{1/p'}
\left(\int_{\R^N}\int_{B_\delta(\mathbf 0)}
\frac{|\v(\x)-\v(\x-\y)|^p}{|\y|^{N+\alpha p}}\,d\x\, d\y\right)^{1/p}\\
&\quad +\int_{\R^N}\int_{B_\delta(\mathbf 0)}g(\y)\big(|\u(\x)|+|\u(\x-\y)|\big)
|\v(\x)|\,d\x\, d\y,
\end{align*}
namely
\begin{equation}\begin{split}\label{jfierhgiub}
\int_{\R^N}& (K\u)(\x)\v(\x)\,d\x\\
\le&\frac{\kappa p}{2} \left(\int_{\R^N}\int_{B_\delta(\mathbf 0)}
\frac{|\u(\x)-\u(\x-\y)|^{p}}{|\y|^{N+\alpha p}}\,d\x\, d\y\right)^{1/p'}\times\\
&\qquad\qquad\times\left(\int_{\R^N}\int_{B_\delta(\mathbf 0)}
\frac{|\v(\x)-\v(\x-\y)|^p}{|\y|^{N+\alpha p}}\,d\x\, d\y\right)^{1/p}\\
&\quad +\int_{\R^N}\int_{B_\delta(\mathbf 0)}g(\y)\big(|\u(\x)|+|\u(\x-\y)|\big)
|\v(\x)|\,d\x\, d\y,
\end{split}\end{equation}

Now, we observe that, for any~$\lambda>0,\,\x,\,\y\in\R^N$, 
\begin{align*} 
g(\y) |\u(\x-\y)|\,|\v(\x)|=&
2g(\y)\frac{\lambda}{\sqrt{2}} |\u(\x-\y)|\, \frac{1}{\lambda\sqrt{2}}\,|\v(\x)|\\
\le& \frac{\lambda^2}{{2}}g(\y)\, |\u(\x-\y)|^2+ \frac{1}{2\lambda^2}g(\y)\,|\v(\x)|^2,
\end{align*}
where we used the Young inequality.
Therefore, integrating the last inequality in~$\R^N\times B_\delta(\mathbf 0)$, we get
\begin{align*}
\int_{\R^N}&\int_{B_\delta(\mathbf 0)}g(\y)\,|\u(\x-\y)|\,|\v(\x)|\,d\x\, d\y
\\\le& \frac{\lambda^2}2 \int_{\R^N}\int_{B_\delta(\mathbf 0)}g(\y)\,|\u(\x-\y)|^2\,d\x\, d\y
+ \frac{1}{2\lambda^2}\int_{\R^N}\int_{B_\delta(\mathbf 0)} g(\y)\,|\v(\x)|^2\,d\x\, d\y
\\=& \frac{\lambda^2}2 \left(\int_{B_\delta(\mathbf 0)}g(\y)\,d\y\right)\,
\left(\int_{\R^N}|\u(\x)|^2\,d\x \right)
+ \frac{1}{2\lambda^2}\left(\int_{B_\delta(\mathbf 0)} g(\y)\,d\y\right)\,
\left(\int_{\R^N}|\v(\x)|^2\,d\x\right).
\end{align*}
Hence, the choice
$$ \lambda:=\left(\frac{\displaystyle \int_{\R^N}|\u(\x)|^2\,d\x}{\displaystyle \int_{\R^N}|\v(\x)|^2\,d\x}
\right)^{\frac14}$$
allows to state
\begin{align*}
\int_{\R^N}\int_{B_\delta(\mathbf 0)}&g(\y)\,|\u(\x-\y)|\,|\v(\x)|\,d\x\, d\y\\
\le& \left(\int_{B_\delta(\mathbf 0)}g(\y)\,d\x\right)\left(\int_{\R^N}|\u(\x)|^2 \,d\x\right)^{1/2}\left(\int_{\R^N}|\v(\x)|^2\,d\x \right)^{1/2}.
\end{align*}
Plugging this information into~\eqref{jfierhgiub}, we conclude that
\begin{align*}
\int_{\R^N}& (K\u)(\x)\v(\x)\,d\x\\
\le &\frac{\kappa p}{2} \left(\int_{\R^N}\int_{B_\delta(\mathbf 0)}
\frac{|\u(\x)-\u(\x-\y)|^{p}}{|\y|^{N+\alpha p}}\,d\x\, d\y\right)^{1/p'}
\left(\int_{\R^N}\int_{B_\delta(\mathbf 0)}
\frac{|\v(\x)-\v(\x-\y)|^p}{|\y|^{N+\alpha p}}\,d\x\, d\y\right)^{1/p}\\
& + 2\left(\int_{B_\delta(\mathbf 0)}g(\y)\,d\x\right)
\left(\int_{\R^N}|\u(\x)|^2 \,d\x\right)^{1/2}\left(\int_{\R^N}|\v(\x)|^2\,d\x \right)^{1/2}.
\end{align*}
Since~$\u$, $\v\in\W$, this implies~\eqref{eq:Sobolev3}.
\smallskip

Now suppose that~\eqref{eq:Sobolev4} holds true and we prove~\eqref{eq:Sobolev5}.
For this, let~$\v\in C^\infty(\R^{N};\R^N)$
be such that every component has compact support.
For the sake of simplicity, we use the notation
\begin{equation}\begin{split}\label{86yjfdbfdjs}
 \U(\x,\y):=&\u(\x)-\u(\x-\y),\\
 \U_n(\x,\y):=&\u_n(\x)-\u_n(\x-\y)\\
 \V(\x,\y):=&\v(\x)-\v(\x-\y).
\end{split}\end{equation}
Arguing as before, we have that
\begin{equation}\begin{split}\label{gjhg5844000}
\int_{\R^N}& \big((K\u)(\x)-(K\u_n)(\x)\big)\v(\x)\,d\x\\
=&-\frac{\kappa p}{2} \underbrace{\int_{\R^N}
\int_{B_\delta(\mathbf 0)}\frac{(\U(\x,\y)-\U_n(\x,\y))|\U(\x,\y)|^{p-2}}{
|\y|^{N+\alpha p}}\V(\x,\y)\,d\x \,d\y}_{\mathcal{A}_n}\\
&-\frac{\kappa p}{2} \underbrace{\int_{\R^N}\int_{B_\delta(\mathbf 0)}
\U(\x,\y)\frac{|\U(\x,\y)|^{p-2}-|\U_n(\x,\y)|^{p-2}}{|\y|^{N+\alpha p}}
\V(\x,\y)\,d\x\, d\y}_{\mathcal{B}_n}\\
&-\underbrace{\int_{\R^N}\int_{B_\delta(\mathbf 0)}\Big(\nabla_{\u}
\Psi(\y,\U(\x,\y))-\nabla_{\u}\Psi(\y,\U_n(\x,\y))\Big)\v(\x)\,d\x\, d\y
}_{\mathcal{C}_n}.
\end{split}\end{equation}
Now we claim that
\begin{equation}\label{VIT}
\mathcal{A}_n\to0.\end{equation}
To check this, we use that each component of~$\v$ is compactly supported, and
we suppose that the support is contained in some ball~$B_R(\mathbf 0)$.
Then,
\begin{align*}  &\frac12\int_{B_{R+\delta}(\mathbf 0)}
\int_{B_\delta(\mathbf 0)} |\U_n(\x,\y)-\U(\x,\y)|^2\,d\x\,d\y
\\\qquad&=\frac12
\int_{B_{R+\delta}(\mathbf 0)}\int_{B_\delta(\mathbf 0)} \big|\big(\u_n(\x)-\u(\x)\big)
-\big(\u_n(\x-\y)-\u(\x-\y)\big)\big|^2\,d\x\,d\y\\
\qquad&\le
\int_{B_{R+\delta}(\mathbf 0)}\int_{B_\delta(\mathbf 0)} |\u_n(\x)-\u(\x)|^2\,d\x\,d\y+
\int_{B_{R+\delta}(\mathbf 0)}\int_{B_\delta(\mathbf 0)}
|\u_n(\x-\y)-\u(\x-\y)|^2\,d\x\,d\y,
\end{align*}
which converges to~$0$ as~$n\to+\infty$, thanks to Lemma~\ref{eq:Sobolev2}.
In particular, up to a subsequence, we can assume that~$\U_n\to\U$
a.e. in~$B_{R+\delta}(\mathbf 0)\times B_\delta(\mathbf 0)$ as~$n\to+\infty$.
Consequently,
$$ \F_n(x,y):=
\frac{\U_n(\x,\y)|\U(\x,\y)|^{p-2}}{
|\y|^{N+\alpha p}}\V(\x,\y)\to
\F(x,y):=
\frac{\U(\x,\y)|\U(\x,\y)|^{p-2}}{
|\y|^{N+\alpha p}}\V(\x,\y)$$
a.e. in~$B_{R+\delta}(\mathbf 0)\times B_\delta(\mathbf 0)$ as~$n\to+\infty$.
We also observe that
\begin{align*} \int_{B_{R+\delta}(\mathbf 0)}&\int_{B_\delta(\mathbf 0)} |\F(x,y)|\,d\x\,d\y\le
\int_{B_{R+\delta}(\mathbf 0)}\int_{B_\delta(\mathbf 0)}
\frac{|\U(\x,\y)|^{p-1}}{
|\y|^{N+\alpha p}}|\V(\x,\y)|\,d\x\,d\y\\ \le&
\left(
\int_{B_{R+\delta}(\mathbf 0)}\int_{B_\delta(\mathbf 0)}
\frac{|\U(\x,\y)|^{p}}{
|\y|^{N+\alpha p}}\,d\x\,d\y\right)^{\frac{p-1}p}
\left(
\int_{B_{R+\delta}(\mathbf 0)}\int_{B_\delta(\mathbf 0)}
\frac{|\V(\x,\y)|^{p}}{
|\y|^{N+\alpha p}}\,d\x\,d\y\right)^{\frac1p}
<+\infty,\end{align*}
where we have used the H\"older Inequality with exponents~$\frac{p}{p-1}$ and~$p$.
Consequently,~$\F\in L^1\big( B_{R+\delta}(\mathbf 0)\times B_\delta(\mathbf 0)\big)$,
and therefore~$\F$ is finite a.e. in~$B_{R+\delta}(\mathbf 0)\times B_\delta(\mathbf 0)$.

We claim that
\begin{equation}\label{UNI F}
{\mbox{$\F_n$ is uniformly integrable.}}\end{equation} To prove this, fix~$\eps>0$
and let~$\eta_\eps>0$ be such that for any measurable~$E\subset
B_{R+\delta}(\mathbf 0)\times B_\delta(\mathbf 0)$ with measure less than~$\eta_\eps$
we have that
$$ \iint_{E}
\frac{|\V(\x,\y)|^{p}}{
|\y|^{N+\alpha p}}\,d\x\,d\y\le\eps.$$
Then, exploiting the H\"older Inequality with exponents~$p$, $\frac{p}{p-2}$ and~$p$,
\begin{align*}
\iint_{E}
\F_n(\x,\y)\,d\x\,d\y \le&\iint_{E}
\frac{|\U_n(\x,\y)|
}{
|\y|^{\frac{N+\alpha p}p}}\cdot
\frac{|\U(\x,\y)|^{p-2}}{
|\y|^{\frac{(N+\alpha p)(p-2)}{p}}}
\cdot
\frac{|\V(\x,\y)|}{
|\y|^{\frac{N+\alpha p}{p}}}
\,d\x\,d\y\\
\le&\left[\iint_{E}
\frac{|\U_n(\x,\y)|^p
}{
|\y|^{{N+\alpha p}}}
\,d\x\,d\y\right]^{\frac{1}{p}}
\left[
\iint_{E}
\frac{|\U(\x,\y)|^{p}}{
|\y|^{N+\alpha p}}
\,d\x\,d\y\right]^{\frac{p-2}{p}}
\left[
\iint_{E}
\frac{|\V(\x,\y)|^p}{
|\y|^{{N+\alpha p}}}
\,d\x\,d\y\right]^{\frac{1}{p}}\\
\le& C\, \eps^{\frac1p},
\end{align*}
for some~$C>0$. This establishes~\eqref{UNI F}. {F}rom it, using Vitali Convergence Theorem,
we obtain~\eqref{VIT}.

Similarly, one can prove that
\begin{equation*}
\mathcal{B}_n\to0\quad {\mbox{ and }}\quad \mathcal{C}_n\to0
\end{equation*}
as~$n\to+\infty$.
Using these pieces of information together with~\eqref{gjhg5844000}
we obtain the desired result \eqref{eq:Sobolev5}.
\end{proof}

\subsection{Definition of weak solutions and main results}\label{sub:def}

In order to look for solutions of~\eqref{eq:CP}, we need to introduce
a suitable functional setting. For this, we denote by~$\mathcal{X}$ the functional
space defined as follows:
\begin{equation}
\label{eq:Sobolev6}
\mathcal{X}:=\left\{\u:\R^N\times[0,\infty)\to\R^N;\>\> \begin{aligned} &\u \in L^\infty(0,T;\W),\,T>0\\
&\pt \u\in L^\infty(0,\infty;L^2(\R^N;\R^N))\end{aligned}\right\}.
\end{equation}
With this, we can give the following definition of weak solutions:

\begin{definition}
\label{def:sol}
Let $\u:\R^N\times[0,\infty)\to\R^N$.
We say that $\u$ is a weak solution of the Cauchy problem \eqref{eq:CP} if
\begin{Definitions}
\item \label{def:1} $\u\in \mathcal{X}$;
\item \label{def:3} for every test function $\v\in C^\infty(\R^{N+1};\R^N)$ such
that every component has compact support, it holds that
\begin{equation}
\label{eq:def}
\begin{split}
\int_0^\infty\int_{\R^N}&\Big(\u(\x,t)\cdot\ptt \v(\x,t)-(K\u(\cdot,t))(\x)\cdot 
\v(\x,t)\Big)\,dt \, d\x\\
&-\int_{\R^N}\v_0(\x)\cdot\v(\x,0)\,d\x+\int_{\R^N}\u_0(\x)\cdot\pt\v(\x,0)\,d\x=0.
\end{split}
\end{equation}
\end{Definitions}
\end{definition}

We can therefore state the main
results of this paper:

\begin{theorem}[{\bf Existence}]
\label{th:existence}
Let~\ref{ass:f.1}, \ref{ass:f.2}, and \ref{ass:f.5} be satisfied.
Then, for every initial datum~$(\u_0,\,\v_0)$ such that
\begin{equation}
\label{eq:assinit}\begin{split}
&\u_0\in L^2(\R^N;\R^N),\,\v_0\in L^2(\R^N;\R^N),\\
{\mbox{and }}\quad & \int_{\R^N}\int_{B_\delta(\mathbf 0)}\Phi
(\y,\u_0(\x)-\u_0(\x-\y))\,d\x\, d\y<\infty,
\end{split}\end{equation}
the Cauchy problem \eqref{eq:CP} admits at least one weak solution in the sense of Definition \ref{def:sol} such that
\begin{equation}
\label{eq:energydecay}
E[\u](t)\le E[\u](0),\qquad \text{for a.e.}\>t\ge0.
\end{equation}
\end{theorem}

\begin{theorem}[{\bf Uniqueness and Stability}]
\label{th:stability}
Let~\ref{ass:f.1}, \ref{ass:f.2}, and \ref{ass:f.5} be satisfied, and let
\begin{equation}\label{p2}
p=2.\end{equation}
If $\u$ and $\wu$ are weak solutions of~\eqref{eq:CP} obtained in
correspondence of the initial data~$(\u_0,\,\v_0)$ and~$(\wu_0,\,\wv_0)$,
respectively, satisfying \eqref{eq:assinit}, then
the following stability estimate holds true:
\begin{equation}
\label{eq:stability}
\begin{split}
&\norm{\pt\u(\cdot,t)-\pt\wu(\cdot,t)}_{L^2(\R^N)}^2\\
&\quad\quad\quad+\kappa\int_{\R^N}\int_{B_\delta(\mathbf 0)}
\frac{|\u(\x,t)-\u(\x-\y,t)-\wu(\x,t)+\wu(\x-\y,t)|^2}{|\y|^{N+2\alpha}}\,d\x\, d\y\\
&\quad\quad \le e^{\left(\lambda+\frac1\kappa\right)t}
\norm{\v_0-\wv_0}_{L^2(\R^N)}^2\\
&\quad\quad\quad+\kappa e^{\left(\lambda+\frac1\kappa\right)t}
\int_{\R^N}\int_{B_\delta(\mathbf 0)}\frac{|\u_0(\x)-\u_0(\x-\y)-\wu_0(\x)
+\wu_0(\x-\y)|^2}{|\y|^{N+2\alpha}}\,d\x\, d\y
\end{split}
\end{equation}
for every $t>0$, where
\begin{equation*}
\lambda:=\int_{B_\delta(\mathbf 0)} g^2(\y)|\y|^{N+2\alpha}\,d\y,
\end{equation*}
and $\kappa$ is the one appearing in \ref{ass:f.5}.
\end{theorem}

In the forthcoming Sections~\ref{sec:exist} and~\ref{sec:stab}
we will give the proofs of Theorems~\ref{th:existence}
and~\ref{th:stability}, respectively.


\section{Proof of Theorem \ref{th:existence}}
\label{sec:exist}

In this section we prove
Theorem \ref{th:existence}. 
The arguments rely on
the compactness of the solutions of suitable approximations 
of~\eqref{eq:CP}.

More precisely, let~$\eps>0$
and~$\ue$ be the unique smooth solution of the fourth order problem
\begin{equation}
\label{eq:CPeps}
\begin{cases}
\ptt \ue(\x,t)=(K\ue(\cdot,t))(\x)-\eps\Delta^2\ue,&\quad \x\in \R^N,\>t>0,\\
\ue(\x,0)=\u_{0,\eps}(\x),&\quad \x\in\R^N,\\
\pt \ue(\x,0)=\v_{0,\eps}(\x),&\quad \x\in\R^N,
\end{cases}
\end{equation}
where~$\u_{0,\eps}$ and~$\v_{0,\eps}$ are smooth approximations
of~$\u_{0}$ and~$\v_{0}$, respectively, such that
\begin{equation}
\label{eq:assinitial_eps}
\begin{split}
&\u_{0,\eps},\,\v_{0,\eps} \in C^\infty(\R^N;\R^N),\quad {\mbox{ for any }} \eps>0,\\
&\u_{0,\eps}\to\u_0,\>\v_{0,\eps}\to\v_0\quad\text{a.e. in $\R^N$ and in $L^2(\R^N;\R^N)$ as $\eps\to0$},\\
&\lim\limits_{\eps\to0}\int_{\R^N}
\int_{B_\delta(\mathbf 0)}\Phi\big(\y,\u_{0,\eps}(\x)-\u_{0,\eps}(\x-\y)\big)\,d\x\, d\y\\
&\qquad\qquad\qquad\qquad=
\int_{\R^N}\int_{B_\delta(\mathbf 0)}\Phi\big(\y,\u_{0}(\x)-\u_{0}(\x-\y)\big)\,d\x\, d\y,\\
&\lim\limits_{\eps\to0}\sqrt{\eps}\norm{\Delta \u_{0,\eps}}_{L^2(\R^N)}=0.
\end{split}
\end{equation}
The well-posedness of \eqref{eq:CPeps}, hence the existence of smooth solutions for that problem, follows
by classical semigroup based arguments, see e.g.~\cite{EN}.
\smallskip

Recalling the notation introduced in Subsections~\ref{subsec:math.2}
and~\ref{sub:def}, 
the main compactness result of this section is the following:

\begin{lemma}
\label{lm:exist}
Let~\ref{ass:f.1}, \ref{ass:f.2} and \ref{ass:f.5} be satisfied.
Then, there exist a sequence $\{\eps_k\}_k\subset(0,\infty)$
and a function~$\u\in\mathcal{X}$ such that, as $k\to\infty$,
\begin{align}
\label{eq:exist1}
&\u_{\eps_k}\to\u \quad\text{a.e. in $\R^N\times[0,\infty)$ and 
in $L^2_{\rm{loc}}(\R^N\times(0,\infty);\R^N)$,}\\
\label{eq:exist1.1}
&\pt \u_{\eps_k}\weak\pt\u \quad\text{in $L^r(0,T;L^2(\R^N;\R^N))$, for any 
$1\le r<\infty$, and~$T>0$,}\\
\label{eq:exist1.2}
&\u_{\eps_k}\weak\u \quad\text{a.e. in $\R^N\times[0,\infty)$
and in $L^r(0,T;\W)$, for any $1\le r<\infty$ and $T>0$,}\\
\label{eq:exist2}
&\text{$\u$ is a weak solution of \eqref{eq:CPeps} in the sense of Definition \ref{def:sol}}.
\end{align}
\end{lemma}

In order to prove Lemma~\ref{lm:exist} we need the following
preliminary results:

\begin{lemma}[{\bf Energy estimate}]
\label{lm:enest}
Let \ref{ass:f.1}, \ref{ass:f.2} and \ref{ass:f.5} be satisfied.
Then, the following formula holds true: 
\begin{equation}
\label{eq:enest}
\begin{split}
&\frac{\norm{\pt \ue(\cdot,t)}^2_{L^2(\R^N)}
+\eps\norm{\Delta \ue(\cdot,t)}^2_{L^2(\R^N)}}{2}\\
&\quad\quad\quad\quad\quad
+\int_{\R^N}\int_{B_\delta(\mathbf 0)}\Phi\big(\y,\ue(\x,t)-\ue(\x-\y,t)\big)\,d\x\, d\y\\
&\quad=  \frac{\norm{\v_{0,\eps}}^2_{L^2(\R^N)}
+\eps\norm{\Delta \u_{0,\eps}}^2_{L^2(\R^N)}}{2}\\
&\quad\quad\quad\quad\quad
+\frac12\int_{\R^N}\int_{B_\delta(\mathbf 0)}\Phi\big(\y,\u_{0,\eps}(\x)-\u_{0,\eps}(\x-\y)\big)\,d\x\, d\y\le C,
\end{split}
\end{equation}
for every $t\ge0$ and for some constant $C>0$ independent on $\eps$.
\end{lemma}

\begin{proof}
Multiplying the equation in~\eqref{eq:CPeps} by~$\pt\ue$,
integrating over~$\R^N$ and recalling~\eqref{eq:operator}, we get
\begin{equation}\begin{split}\label{rey5bg}
0=&\int_{\R^N} \ptt \ue \pt \ue \,d\x-
\int_{\R^N}(K\ue(\cdot,t))(\x)\pt \ue\, d\x+\eps\int_{\R^N}\Delta^2\ue \pt \ue \,d\x\\
=&\int_{\R^N} \ptt \ue \pt \ue \,d\x-
\int_{\R^N}(K\ue(\cdot,t))(\x)\pt \ue\, d\x
+\eps\int_{\R^N}\Delta\ue \pt\Delta \ue \,d\x\\
=&\frac{d}{dt}\int_{\R^N}\frac{|\pt \ue|^2+\eps|\Delta\ue|^2}{2}\,d\x\\
&+\int_{\R^N} \int_{B_\delta(\mathbf 0)}\f(\y,\ue(\x,t)-\ue(\x-\y,t))
\pt\ue(\x,t)\, d\x \,d\y
.\end{split}\end{equation}
Now we use~\ref{ass:f.2} to see that 
\begin{align*}
\int_{\R^N} &\int_{B_\delta(\mathbf 0)}\f(\y,\ue(\x,t)-\ue(\x-\y,t))
\pt\ue(\x,t)\, d\x \,d\y\\
=& \frac{1}{2}\int_{\R^N} \int_{B_\delta(\mathbf 0)}\f(\y,\ue(\x,t)-\ue(\x-\y,t))\pt\ue(\x,t)\, d\x\, d\y\\
&+\frac{1}{2}\int_{B_\delta(\mathbf 0)}\left(\int_{\R^N}
\f(\y,\ue(\x,t)-\ue(\x-\y,t))\pt\ue(\x,t)\, d\x\right)\, d\y\\
=&\frac{1}{2}\int_{\R^N} \int_{B_\delta(\mathbf 0)}\f(\y,\ue(\x,t)-\ue(\x-\y,t))
\pt\ue(\x,t) \,d\x\, d\y\\
& +\frac{1}{2}\int_{B_\delta(\mathbf 0)}\left(\int_{\R^N} 
\f(\y,\ue(\z+\y,t)-\ue(\z,t))\pt\ue(\z+\y,t)\, d\z\right) \,d\y\\
=&\frac{1}{2}\int_{\R^N} \int_{B_\delta(\mathbf 0)}\f(\y,\ue(\x,t)
-\ue(\x-\y,t))\pt\ue(\x,t)\, d\x\, d\y\\
& +\frac{1}{2}\int_{\R^N}\left(\int_{B_\delta(\mathbf 0)} \f(-\y,\ue(\z-\y,t)
-\ue(\z,t))\pt\ue(\z-\y,t)\, d\y\right)\, d\z\\
=&
\frac{1}{2}\int_{\R^N} \int_{B_\delta(\mathbf 0)}\f(\y,\ue(\x,t)-
\ue(\x-\y,t))\pt\ue(\x,t) \,d\x\, d\y\\
& -\frac{1}{2}\int_{\R^N}\left(\int_{B_\delta(\mathbf 0)} \f(\y,\ue(\z,t)
-\ue(\z-\y,t))\pt\ue(\z-\y,t)\, d\y\right)\, d\z\\
=& \frac{1}{2}\int_{\R^N} \int_{B_\delta(\mathbf 0)}
\f(\y,\ue(\x,t)-\ue(\x-\y,t))\Big(\pt\ue(\x,t)-\pt\ue(\x-\y,t)\Big)\, d\x\, d\y.
\end{align*}
Plugging this information into~\eqref{rey5bg}, we conclude that
\begin{align*}
0=&\frac{d}{dt}\int_{\R^N}\frac{|\pt \ue|^2+\eps|\Delta\ue|^2}{2}\,d\x\\
& +\frac{1}{2}\int_{\R^N} \int_{B_\delta(\mathbf 0)}
\f(\y,\ue(\x,t)-\ue(\x-\y,t))\Big(\pt\ue(\x,t)-\pt\ue(\x-\y,t)\Big)\, d\x\, d\y.
\end{align*}
As a consequence, using~\ref{ass:f.5},
\begin{align*}
0=&\frac{d}{dt}\int_{\R^N}\frac{|\pt \ue|^2+\eps|\Delta\ue|^2}{2}\,d\x\\
& +\frac{1}{2}\int_{\R^N} \int_{B_\delta(\mathbf 0)}
\nabla_u\Phi(\y,\ue(\x,t)-\ue(\x-\y,t))\Big(\pt\ue(\x,t)-\pt\ue(\x-\y,t)\Big)\, d\x\, d\y\\
=&\frac{d}{dt}\left[\int_{\R^N}\frac{|\pt \ue|^2+\eps|\Delta\ue|^2}{2}\,d\x
+\frac{1}{2}\int_{\R^N} \int_{B_\delta(\mathbf 0)}\Phi(\y,\ue(\x,t)-\ue(\x-\y,t))
\,d\x\, d\y\right].
\end{align*}
Hence, an integration over $(0,t)$ gives the desired equality in~\eqref{eq:enest}.
Furthermore, the boundedness of the quantity in~\eqref{eq:enest}
follows from the convergence assumptions in~\eqref{eq:assinitial_eps}.
The proof of Lemma~\ref{lm:enest} is thus complete.
\end{proof}

\begin{lemma}[{\bf $L^2-$estimate}]
\label{lm:l2}
Let~\ref{ass:f.1}, \ref{ass:f.2}, and \ref{ass:f.5} be satisfied.
Then, the following estimate holds true:
\begin{equation}
\label{eq:l2}
\norm{\ue(\cdot,t)}_{L^2(\R^N)}\le C(1+t),
\end{equation}
for every $t\ge0$ and for some constant $C>0$
independent on~$\eps$.
\end{lemma}

\begin{proof}
We observe that
\begin{equation}\label{ieghegb57t7000}\begin{split}
|\ue (\x,t)|\le&|\u_{0,\eps}(\x)|+\int_0^t |\pt \ue(\x,s)|\,ds\\
\le&|\u_{0,\eps}(\x)|+\sqrt{t}\sqrt{\int_0^t |\pt \ue(\x,s)|^2\,ds}.\end{split}
\end{equation}
Taking the square of both sides of~\eqref{ieghegb57t7000}
and integrating over $\R^n$, we have that
\begin{align*}
\int_{\R^N}|\ue (\x,t)|^2\,d\x\le&\>
2\int_{\R^N}|\u_{0,\eps}(\x)|^2\,d\x+2t\int_0^t \int_{\R^N}|\pt \ue(\x,s)|^2\,ds\,d\x\\
\le&\> 2\norm{\u_{0,\eps}}_{L^2(\R^N)}^2+
2t^2\sup_{t\ge0}\norm{\pt \ue(\cdot,t)}_{L^2(\R^N)}^2.
\end{align*}
Therefore, the desired
estimate follows from~\eqref{eq:assinitial_eps} and~\eqref{eq:enest}.
\end{proof}

\begin{proof}[Proof of Lemma \ref{lm:exist}]
We notice that, by virtue of Lemma~\ref{lm:enest},
\begin{equation}\label{orgeirhgrhr1}
\text{$\{\pt\ue\}_\eps$ is a bounded sequence
in $L^\infty(0,\infty;L^2(\R^N;\R^N))$}.\end{equation}
Furthermore, using again
Lemmas~\ref{lm:enest} and~\ref{lm:l2} and 
assumption~\ref{ass:f.5} we obtain that
\begin{equation}\label{orgeirhgrhr2}
\text{$\left\{\ue\right\}_\eps$ is a bounded sequence
in $L^\infty(0,T;\W)$ for every $T>0$}.
\end{equation}
Therefore, by Lemma~\ref{eq:Sobolev2}, we have that
there exist a subsequence~$\{\u_{\eps_k}\}_k$ and a
function~$\u\in L^2_{\rm{loc}}(\R^N;\R^N)$
such that~\eqref{eq:exist1} holds true.

As a matter of fact, by virtue of~\eqref{orgeirhgrhr1} and~\eqref{orgeirhgrhr2}
we have that~$\u\in L^\infty(0,T;\W)$ for every $T>0$,
and~$\{\pt\ue\}_\eps\in L^\infty(0,\infty;L^2(\R^N;\R^N))$.
That is, recalling the definition of~$\mathcal{X}$ in~\eqref{eq:Sobolev6},
the function~$\u\in \mathcal{X}$. 
Thus, condition~\ref{def:1} in
Definition~\ref{def:sol} holds true for $\u$.

Hence, to prove \eqref{eq:exist2}, we now focus on proving that $\u$ satisfies ~\ref{def:3} in
Definition~\ref{def:sol}.
To this aim, let~$\v\in C^\infty(\R^{N+1};\R^N)$
be a given test function such that every component has compact support. 
Multiplying~\eqref{eq:CPeps} by~$\v$ and integrating
over~$(0,\infty)\times \R^N$, we get
\begin{equation}\begin{split}\label{reh45y12135476}
\int_0^\infty\int_{\R^N}&
\left(\u_{\eps_k}(\x,t)\cdot\ptt \v(\x,t)-
(K\u_{\eps_k}(\cdot,t))(\x)\cdot \v(\x,t)\right)\,dt\,d\x\\
& -\int_{\R^N}\v_{0,\eps_k}(\x)\cdot\v(\x,0)\,d\x
+\int_{\R^N}\u_{0,\eps_k}(\x)\cdot\pt\v(\x,0)\,d\x\\
=&\;-\eps_k\int_0^\infty\int_{\R^N}\u_{\eps_k}(\x,t)\cdot\Delta^2 \v(\x,t)\,dt\,d\x.
\end{split}\end{equation}
Hence, sending~$k\to\infty$ in~\eqref{reh45y12135476} and using
Lemma~\ref{lm:Sobolev1} and formula~\eqref{eq:assinitial_eps},
we obtain~\ref{def:3}. 

Therefore, $\u$ is a weak solution
of~\eqref{eq:CPeps} in the sense of Definition \ref{def:sol}. 
This proves~\eqref{eq:exist2},
and so the proof of Lemma \ref{lm:exist} is complete.
\end{proof}

\begin{proof}[Proof of Theorem \ref{th:existence}]
Thanks to Lemma \ref{lm:exist},
in order to complete the proof of Theorem \ref{th:existence},
it remains to prove that~$\u$ satisfies~\eqref{eq:energydecay}.
For this,  we use~\ref{ass:f.5} and~\eqref{eq:enest} to see that
\begin{equation}
\label{eq:endec1}
\begin{split}
&\frac{\norm{\pt \uek(\cdot,t)}^2_{L^2(\R^N)}}{2}
+\frac{\kappa}2\int_{\R^N}\int_{B_\delta(\mathbf 0)}
\frac{|\uek(\x)-\uek(\x-\y)|^{p}}{|\y|^{N+\alpha p}}\,d\x\, d\y\\
&\qquad \qquad+\frac12 \int_{B_R(\mathbf 0)}\int_{B_\delta(\mathbf 0)}\Psi
\big(\y,\uek(\x,t)-\uek(\x-\y,t)\big)\,d\x\, d\y\\
&\qquad\le  \frac{\norm{\v_{0,\epk}}^2_{L^2(\R^N)}+
\epk\norm{\Delta \u_{0,\epk}}^2_{L^2(\R^N)}}{2}\\
&\qquad \qquad+\frac12 \int_{\R^N}\int_{B_\delta(\mathbf 0)}
\Phi\big(\y,\u_{0,\epk}(\x)-\u_{0,\epk}(\x-\y)\big)\,d\x\, d\y,
\end{split}
\end{equation}
for every $R>0$.
Now, making use of~\eqref{eq:assinitial_eps} and recalling~\eqref{eq:enert},
we can say that
\begin{equation}
\label{eq:endec2}
\begin{split}
\lim_{k}&\left(\frac{\norm{\v_{0,\epk}}^2_{L^2(\R^N)}
+\epk\norm{\Delta \u_{0,\epk}}^2_{L^2(\R^N)}}{2}\right.\\
&\qquad \qquad\left.
+\int_{\R^N}\int_{B_\delta(\mathbf 0)}\Phi(\y,\u_{0,\epk}(\x)
-\u_{0,\epk}(\x-\y))\,d\x\, d\y\right)\\
=&
\frac{\norm{\pt \u(\cdot,t)}^2_{L^2(\R^N)}}{2}
+\frac{1}{2}\int_{\R^N}\int_{B_\delta(\mathbf 0)}
\Phi(\y,\u(\x,t)-\u(\x-\y,t))\,d\x\, d\y=E[\u](0).
\end{split}
\end{equation}
Moreover, the Dominated Convergence Theorem,
\ref{ass:f.5} and~\eqref{eq:exist1} give that, for a.e.~$t>0$,
\begin{equation}
\label{eq:endec3}
\begin{split}
\lim_{k}&\norm{\pt \uek(\cdot,t)}^2_{L^2(\R^N)}\ge \norm{\pt \u(\cdot,t)}^2_{L^2(\R^N)},\\
\lim_{k}&\int_{B_R(\mathbf 0)}
\int_{B_\delta(\mathbf 0)}\Psi(\y,\uek(\x,t)-\uek(\x-\y,t))\,d\x\, d\y\\
&=\int_{B_R(\mathbf 0)}\int_{B_\delta(\mathbf 0)}\Psi(\y,\u(\x,t)-\u(\x-\y,t))\,d\x\, d\y,\\
\liminf_{k}&\int_{\R^N}\int_{B_\delta(\mathbf 0)}\frac{|\uek(\x)-\uek(\x-\y)|^{p}}{|\y|^{N+\alpha p}}d\x d\y\\
&\ge 
\int_{\R^N}\int_{B_\delta(\mathbf 0)}\frac{|\u(\x)-\u(\x-\y)|^{p}}{
|\y|^{N+\alpha p}}\,d\x\, d\y.
\end{split}
\end{equation}
Therefore, sending $k\to\infty$ in \eqref{eq:endec1} we get
\begin{equation}
\label{eq:endec4}
\begin{split}
\frac{\norm{\pt \u(\cdot,t)}^2_{L^2(\R^N)}}{2}
&+\frac{\kappa}2 \int_{\R^N}\int_{B_\delta(\mathbf 0)}
\frac{|\u(\x)-\u(\x-\y)|^{p}}{|\y|^{N+\alpha p}}\,d\x\, d\y\\
&+\frac{1}{2}\int_{B_R(\mathbf 0)}\int_{B_\delta(\mathbf 0)}
\Psi(\y,\u(\x,t)-\u(\x-\y,t))\,d\x\, d\y
\le E[\u](0).
\end{split}
\end{equation}
Sending $R\to\infty$ in~\eqref{eq:endec4} and recalling again~\eqref{eq:enert},
we gain \eqref{eq:energydecay}, as desired.
\end{proof}

\section{Uniqueness, stability and
proof of Theorem \ref{th:stability}}
\label{sec:stab}

This section is devoted to the proof of Theorem \ref{th:stability}.

\begin{proof}[Proof of Theorem \ref{th:stability}]
Let $\u$ and $\wu$ be two weak solutions of \eqref{eq:CP} 
according to Definition~\ref{def:sol}, and define
\begin{equation}\label{defw00}
\w:=\u-\wu.
\end{equation}
We claim that~$\w$ solves the equation
\begin{equation}
\label{eq:w}
\begin{split}
\ptt \w(\x,t)=&-2\kappa\int_{B_\delta(\mathbf 0)}
\frac{\w(\x,t)-\w(\x-\y,t)}{|\y|^{N+2\alpha}}\,d\y \\
&
-\int_{B_\delta(\mathbf 0)}\int_0^1 {\mathcal{F}}(\theta,\x,\y,t)
\big(\w(\x,t)-\w(\x-\y,t)\big)\,d\y \,d\theta.
\end{split}
\end{equation}
where
\begin{equation}\label{132fhbvhsalxm}
{\mathcal{F}}(\theta,\x,\y,t):=D^2_{\u}\Psi\Big
(\y,\theta\big(\u(\x,t)-\u(\x-\y,t)\big)+(1-\theta)
\big(\wu(\x,t)-\wu(\x-\y,t)\big)\Big).
\end{equation}
To prove~\eqref{eq:w}, we recall~\eqref{p2}
and we use~\eqref{eq:CP}, \eqref{eq:operator},
\eqref{ass:f.2}
and~\ref{ass:f.5} to see that
\begin{align*}
\ptt \w(\x,t)=&\ptt \u(\x,t)-\ptt \wu(\x,t)\\
=&(K\u(\cdot,t))(\x)-(K\wu(\cdot,t))(\x)\\
=& \int_{B_\delta(\mathbf 0)}
\Big(\f(-\y,\u(\x-\y,t)-\u(\x,t))-\f(-\y,\wu(\x-\y,t)-\wu(\x,t))\Big)\,d\y\\
=&-\int_{B_\delta(\mathbf 0)}
\Big(\f(\y,\u(\x,t)-\u(\x-\y,t))-\f(\y,\wu(\x,t)-\wu(\x-\y,t))\Big)\,d\y\\
=&-2\kappa\int_{B_\delta(\mathbf 0)}
\frac{\w(\x,t)-\w(\x-\y,t)}{|\y|^{N+2\alpha}}\,d\y\\
& -\int_{B_\delta(\mathbf 0)}\Big(\nabla_\u\Psi\big(\y,\u(\x,t)-
\u(\x-\y,t)\big)-\nabla_\u\Psi\big(\y,\wu(\x,t)-\wu(\x-\y,t)\big)\Big)\,d\y\\
=&-2\kappa\int_{B_\delta(\mathbf 0)}
\frac{\w(\x,t)-\w(\x-\y,t)}{|\y|^{N+2\alpha}}\,d\y \\
&-\int_{B_\delta(\mathbf 0)}\int_0^1 {\mathcal{F}}
(\theta,\x,\y,t)\big(\w(\x,t)-\w(\x-\y,t)\big)\,d\y\, d\theta,
\end{align*}
where~${\mathcal{F}}$ is defined in~\eqref{132fhbvhsalxm}.
This proves~\eqref{eq:w}. 

Now, we multiply~\eqref{eq:w} by $\pt\w$ and we integrate over~$\R^N$.
In this way, making again use of~\ref{ass:f.2} and~\ref{ass:f.5}, we get
\begin{align*}
0=&\int_{\R^N} \ptt \w \,\pt \w\, d\x
+2\kappa\int_{\R^N}\int_{B_\delta(\mathbf 0)}
\frac{\w(\x,t)-\w(\x-\y,t)}{|\y|^{N+2\alpha}}\,\pt \w(\x,t)\,d\x\, d\y \\
&+\int_{\R^N}\int_{B_\delta(\mathbf 0)}\int_0^1
{\mathcal{F}}(\theta,\x,\y,t)\big(\w(\x,t)-\w(\x-\y,t)\big)\pt \w(\x,t)\,d\x\, d\y\, d\theta\\
=&\int_{\R^N} \ptt \w\, \pt \w\, d\x
+\kappa\int_{\R^N}\int_{B_\delta(\mathbf 0)}
\frac{\w(\x,t)-\w(\x-\y,t)}{|\y|^{N+2\alpha}}\,
\pt \big(\w(\x,t)-\w(\x-\y,t)\big)\,d\x\, d\y \\
&+\int_{\R^N}\int_{B_\delta(\mathbf 0)}
\int_0^1 {\mathcal{F}}(\theta,\x,\y,t)\big(\w(\x,t)-\w(\x-\y,t)\big)
\pt \w(\x,t)\,d\x\, d\y\, d\theta\\
=&\frac{d}{dt}\int_{\R^N}
\frac{|\pt \w|^2}{2}\,d\x+\frac{\kappa}{2}\frac{d}{dt}
\int_{\R^N}\int_{B_\delta(\mathbf 0)}
\frac{|\w(\x,t)-\w(\x-\y,t)|^2}{|\y|^{N+2\alpha}}\,d\x\, d\y \\
&+\int_{\R^N}\int_{B_\delta(\mathbf 0)}
\int_0^1 {\mathcal{F}}(\theta,\x,\y,t)\big(\w(\x,t)-\w(\x-\y,t)\big)
\pt \w(\x,t)\,d\x\, d\y\, d\theta.
\end{align*}
Therefore, thanks to \ref{ass:f.5},
\begin{align*}
\frac{d}{dt}&\left(\int_{\R^N}\frac{|\pt \w|^2}{2}\,d\x
+\frac{\kappa}{2}
\int_{\R^N}\int_{B_\delta(\mathbf 0)}
\frac{|\w(\x,t)-\w(\x-\y,t)|^2}{|\y|^{N+2\alpha}}\,d\x\, d\y\right)\\
=&-\int_{\R^N}\int_{B_\delta(\mathbf 0)}\int_0^1
{\mathcal{F}}(\theta,\x,\y,t)\big(\w(\x,t)-\w(\x-\y,t)\big)\pt \w(\x,t)\,d\x\, d\y\, d\theta\\
\le &\int_{\R^N}\int_{B_\delta(\mathbf 0)} 
g(\y)\big|\w(\x,t)-\w(\x-\y,t)\big|\;|\pt \w(\x,t)|\,d\x\, d\y\\
\le &\frac{1}{2}\underbrace{\left(
\int_{B_\delta(\mathbf 0)} g^2(\y)|\y|^{N+2\alpha}\,d\y\right)}_{\lambda}
\int_{\R^N}|\pt \w(\x,t)|^2\,d\x\\
&+\frac{1}{2}\int_{\R^N}\int_{B_\delta(\mathbf 0)}
\frac{|\w(\x,t)-\w(\x-\y,t)|^2}{|\y|^{N+2\alpha}} \,d\x\, d\y\\
\le &\left(\lambda+\frac1{\kappa}\right)
\left(\int_{\R^N}\frac{|\pt \w(\x,t)|^2}2\,d\x
+\frac{\kappa}{2}\int_{\R^N}\int_{B_\delta(\mathbf 0)}
\frac{|\w(\x,t)-\w(\x-\y,t)|^2}{|\y|^{N+2\alpha}}\, d\x\, d\y\right).
\end{align*}
Hence, applying the Gronwall Lemma, we conclude that
\begin{align*}
\int_{\R^N}&\frac{|\pt \w(\x,t)|^2}2\,d\x
+\frac{\kappa}{2}\int_{\R^N}\int_{B_\delta(\mathbf 0)}
\frac{|\w(\x,t)-\w(\x-\y,t)|^2}{|\y|^{N+2\alpha}}\, d\x\, d\y
\\\le& 
e^{\left(\lambda+\frac1\kappa\right)t}
\left(\int_{\R^N}\frac{|\pt \w(\x,0)|^2}2\,d\x
+\frac{\kappa}{2}\int_{\R^N}\int_{B_\delta(\mathbf 0)}
\frac{|\w(\x,0)-\w(\x-\y,0)|^2}{|\y|^{N+2\alpha}}\, d\x\, d\y\right).
\end{align*}
Consequently, recalling~\eqref{defw00}, we obtain~\eqref{eq:stability},
as desired.
\end{proof}



\begin{thebibliography}{99}

\bibitem{B}
J. M. Ball,
\textit{Convexity conditions and existence theorems in nonlinear elasticity},
Arch. Rational Mech. Anal. 63 (1976/77), no. 4, 337-403.

\bibitem{DPV}
E. Di Nezza, G. Palatucci, E. Valdinoci,
\textit{Hitchhiker's guide to the fractional Sobolev spaces},
Bull. Sci. Math. 136 (2012), no. 5, 521-573.

\bibitem{PAT}
S. Dipierro, S. Patrizi, E. Valdinoci, \textit{Heteroclinic connections for nonlocal
equations},
ArXiv:1711.01491.

\bibitem{DZ}
Q. Du, K. Zhou,
\textit{Mathematical analysis for the peridynamic nonlocal continuum theory},
ESAIM Math. Model. Numer. Anal. 45 (2011), no. 2, 217-234.

\bibitem{ELP} E.  Emmrich, R. B. Lehoucq, D. Puhst,
\textit{Peridynamics: a nonlocal continuum theory}.  
In: Griebel, M., Schweitzer, M.A. (eds.) Meshfree
Methods for Partial Differential Equations VI. Lect. N. Comput. Sci.
Eng., vol. 89. Springer, Berlin (2013).

\bibitem{EP1}
E. Emmrich, D. Puhst, \textit{Well-posedness
of the peridynamic model with Lipschitz continuous pairwise force function},
Commun. Math. Sci. 11 (2013), no. 4, 1039-1049.

\bibitem{EP} E. Emmrich, D. Puhst,
\textit{Survey of Existence Results in Nonlinear
Peridynamics in Comparison with Local
Elastodynamics}, Comput. Methods Appl. Math. 15 (2015), no. 4, 483-496.

\bibitem{EP2}
E. Emmrich, D. Puhst, \textit{Measure-valued and
weak solutions to the nonlinear peridynamic model in nonlocal elastodynamics},
Nonlinearity 28 (2015), 285-307.

\bibitem{EW}
E. Emmrich, O. Weckner, \textit{On the well-posedness
of the linear peridynamic model and its convergence towards the
Navier equation of linear elasticity}, Commun. Math. Sci. 5
(2007), no. 4, 851-864.

\bibitem{EN}
K.-J. Engel, R. Nagel. \textit{One-parameter
semigroups for linear evolution equations},
Graduate Texts in Mathematics, 194. Springer-Verlag, New York (2000).

\bibitem{Er} A. C. Eringen, \textit{Nonlocal Continuum Field Theories},
Springer-Verlag, New York Berlin Heidelberg (2002).

\bibitem{EE}  A. C. Eringen, D. G. B. Edelen,
\textit{On nonlocal elasticity}, Int. J. Eng. Sci. 10 (1972), no. 3, 233-248.

\bibitem{K}  E. Kr\"oner, \textit{Elasticity theory of
materials with long range cohesive forces}, 
Int. J. Solids Structures 3 (1967), 731-742.

\bibitem{Ku} I. A. Kunin, \textit{Elastic Media with Microstructure},
Voll. I, II, Springer-Verlag Berlin Heidelberg (1982).

\bibitem{MD}  T. Mengesha, Q. Du, \textit{Characterization of function spaces of vector fields and an application in nonlinear peridynamics}, 
Nonlinear Anal. 140 (2016), 82-111.

\bibitem{Sid} T. C. Sideris, \textit{Nonresonance and global existence of prestressed nonlinear elastic waves}, 
Ann. of Math. (2) 151 (2000), no. 2, 849-874.

\bibitem{S1} S. A. Silling, \textit{Reformulation of elasticity
theory for discontinuities and long-range forces}, 
J. Mech. Phys. Solids 48 (2000), 175-209.

\bibitem{S} S. A. Silling, \textit{Linearized theory of peridynamic states}, 
J. Elasticity 99 (2010),
85-111.

\bibitem{SEWX} S. A. Silling, M. Epton, O. Weckner, J. Xu, E. Askari,
\textit{Peridynamic states
and constitutive modeling}, J. Elasticity 88 (2007), 151-184.

\bibitem{SL} S. A. Silling, R. B. Lehoucq, \textit{Convergence
of peridynamics to classical elasticity
theory}, J. Elasticity 93 (2008), no. 1, 13-37.

\bibitem{SL1} S. A. Silling, R. B. Lehoucq, \textit{Peridynamic Theory
of Solid Mechanics}, Adv.
Appl. Mech. 44 (2010), 73-168.

\end{thebibliography}
\end{document}